\documentclass[12pt]{iopart}

\usepackage{epstopdf}
\usepackage{subfigure}

\usepackage{latexsym,bm}
\usepackage{xcolor,graphicx}
\usepackage{eso-pic}
\usepackage{caption}
\usepackage{subfigure}
\newtheorem{lemma}{Lemma}[section]
\newtheorem{theorem}{Theorem}[section]

\newtheorem{proof}{Proof}
\newtheorem{remark}{Remark}[section]
\newtheorem{example}{Example}[section]

\makeatletter

\@addtoreset{equation}{section}
\makeatother

\begin{document}

\title{A non-conditional divergence criteria of Petrov-Galerkin method for bounded linear operator equation}
\author{Yidong Luo}
\address{School of Mathematics and Statistics, Wuhan University, Hubei Province, P. R. China}
\ead{Sylois@whu.edu.cn}
\vspace{10pt}
\begin{indented}
\item[]August 2017
\end{indented}

\begin{abstract}
 Petrov-Galerkin methods are always considered in numerical solutions of differential and integral equations $ Ax=b $. It is common to consider the convergence and error analysis when $ b \in \mathcal{R}(A) $ which make the equation solvable. However, the case when $ b \notin \mathcal{R}(A) $ is always ignored. In this paper, we consider the numerical behavior of Petrov-Galerkin methods when $ b \notin \mathcal{R}(A) $. It is a natural guess that when $ b \in \mathcal{R}(A) $, the corresponding approximate solution constructed by Petrov-Galerkin methods with arbitrary basis will diverge to infinity. We prove this conjecture for bounded linear operator equation with dense range $ \mathcal{R}(A) $ and give a more general divergence result for bounded linear operator equation with not necessarily dense range $ \mathcal{R}(A) $. Several applications show its power.
\end{abstract}
%
%
%
%
%

\section{Introduction of divergence analysis and main result}
Let $ X,Y $ be Hilbert spaces over the complex scalar field, $ \mathcal{B}(X,Y) $ and $ \mathcal{C}(X,Y) $ denotes the set of all bounded, compact linear operators  mapping from $ X $ to $ Y $ respectively, $ \{ X_n \} $ and $ \{ Y_n \} $ be sequences of closed subspaces of $ X $ and $ Y $ respectively, $ P_n := P_{X_n} $ and $ Q_n:= Q_{Y_n} $ be orthogonal projection operators which project $ X $ and $ Y $ onto $ X_n $ and $ Y_n $ respectively.

 Let the original operator equation of the first kind be
\begin{equation}
Ax = b, \quad A \in \mathcal{B}(X,Y), \ x \in X, \ b \in \mathcal{R}(A) \oplus \mathcal{R}(A)^\perp
\end{equation}
with Moore-Penrose inverse solution $ x^\dagger := A^\dagger b $. The Petrov-Galerkin approximation setting of (1.1) is
\begin{equation}
A_n x_n = b_n, \ A_n \in \mathcal{B}(X_n,Y_n), \ x_n \in X_n, \ b_n:= Q_n b  \in Y_n,
\end{equation}
where the general Petrov-Galerkin approximation operator $ A_n := Q_n A P_n : X_n \to Y_n $ possesses  closed range $ \mathcal{R}(A_n) $. Without assumptions on the existence or uniqueness of solution to (1.2), we denote the Moore-Penrose inverse solution of (1.2) by $ x^\dagger_n := A^\dagger_n Q_n b $.
\newline \indent It is a common way to analyze the convergence of $ x^\dagger_n $ to $ x^\dagger $ when $ b \in \mathcal{R}(A) $. There is a large amount of literature inside this topic, see [2-6] and references therein. However, in practical computation, there always interfuse noise into $ b $, thus it is disrupted into $ b^\delta := b + \delta b $ which may not locate in $ \mathcal{R}(A) $. This leads to a question: what is the numerical behavior of approximation solution $ x^\dagger_n := A_n^\dagger Q_n b $ when $ b \notin \mathcal{R}(A) $?
\newline \indent The literature on this topic is few. We only find the result in [4]:
\begin{lemma}
For Petrov-Galerkin setting (1.1), (1.2), if $ (\{ X_n \}_{n \in \mathrm{N}}, \{ Y_n \}_{n \in \mathrm{N}} ) $ satisfies the completeness condition, that is,
\begin{equation*}
\hbox{ \raise-2mm\hbox{$\textstyle s-\lim \atop \scriptstyle {n \to \infty}$}} P_n = I_X, \quad
\hbox{ \raise-2mm\hbox{$\textstyle s-\lim \atop \scriptstyle {n \to \infty}$}} Q_n = I_Y,
\end{equation*}
and
\begin{equation}
\sup_n \Vert A^\dagger_n Q_n A \Vert < \infty
\end{equation}
where $ \dagger $ denotes the Moore-Penrose inverse of linear operator (See [1,3]),
then, for $ b \notin \mathcal{D}(A^\dagger) = \mathcal{R}(A) \oplus \mathcal{R}(A)^\perp $,
\begin{equation}
\lim_{n \to \infty} \Vert A^\dagger_n Q_n Q_{\overline{\mathcal{R}(A)}} b \Vert = \infty.
\end{equation}
\end{lemma}
\begin{proof}
See [4, Theorem 2.2 (c)]
\end{proof}
\begin{remark}
In particular, if $ b \in \overline{\mathcal{R}(A)} \setminus \mathcal{R}(A) $, we have $ \lim_{n \to \infty} \Vert A^\dagger_n Q_n b \Vert_{X} = \infty. $ Furthermore, if $ A \in \mathcal{B}(X,Y) $ possess a dense range in $ Y $, then $ \lim_{n \to \infty} \Vert A^\dagger_n Q_n b \Vert_{X} = \infty, \quad (b \in Y \setminus \mathcal{R}(A)) $.
\end{remark}
\begin{remark}
Notice that, for some bounded linear operator $ A $, $ Y \setminus \mathcal{R}(A) $ can be far bigger than $ \mathcal{R}(A)  $. See Example 3.1.
\end{remark}

\indent 
Above result can partially answer to the question proposed before. However, the verification of condition (1.3) is a non-trivial task (See [2, Chapter 3]). In this paper, we prove that, without (1.3), above divergence result also holds:
\begin{theorem}
Let $ A $ be a bounded linear operator mapping from $ X $ to $ Y $, $ X,Y $ be infinite-dimensional Hilbert spaces, $ \{ \xi_k \}^{\infty}_{k=1} $, $ \{ \eta_k \}^\infty_{k=1} $ be arbitrary $ X $ basis , $ Y $ basis respectively, $ \{ X_n \} $, $ \{ Y_n \} $ be corresponding subspace sequence, that is, $ X_n = \textrm{span} \{ \xi_k \}^{n}_{k=1} $, $ Y_n := \textrm{span}\{ \eta_k \}^{n}_{k=1} $. Then the following three properties holds:
 \newline \indent (a) For any $ b \notin \mathcal{D}(A^\dagger) = \mathcal{R}(A) \oplus \mathcal{R}(A)^\perp $,
\begin{equation}
\lim_{n \to \infty} \Vert A^\dagger_n Q_n Q_{\overline{\mathcal{R}(A)}} b \Vert_{X} = \infty,
\end{equation}
where $ A_n := Q_n A P_n: X_n \to Y_n $ is the general Petrov-Galerkin approximation operator. 
\newline \indent (b) In particular, if $ b \in \overline{\mathcal{R}(A)} \setminus \mathcal{R}(A) $, we have
\begin{equation}
\lim_{n \to \infty} \Vert A^\dagger_n Q_n b \Vert_{X} = \infty.
\end{equation}

\indent (c) If $ A \in \mathcal{B}(X,Y) $ possess a dense range in $ Y $, then (1.6) is transformed into
\begin{equation}
\lim_{n \to \infty} \Vert A^\dagger_n Q_n b \Vert_{X} = \infty, \quad (b \in Y \setminus \mathcal{R}(A))
\end{equation}
\end{theorem}

\section{Proof}
Before the proof, we first prepare some technical lemmas.
\begin{lemma}
Let $ A \in \mathcal{B}(X,Y) $ and $ A_n := A P_n : X_n \to Y $, $ \{ X_n \}_{n \in \mathrm{N}} $ satisfies the completeness condition, that is,
\begin{equation*}
\hbox{ \raise-2mm\hbox{$\textstyle s-\lim \atop \scriptstyle {n \to \infty}$}} P_n = I_X,
\end{equation*}
then there holds
\begin{equation}
\mathcal{R}(A_n) = A(X_n) \subseteq \mathcal{R}(A), \ \hbox{ \raise-2mm\hbox{$\textstyle s-\lim \atop \scriptstyle {n \to \infty}$}} Q_{\mathcal{R}(A_n)} = Q_{\overline{\mathcal{R}(A)}}.
\end{equation}
\end{lemma}
\begin{proof}
See [6, Lemma 2.2 (a)].
\end{proof}
Based on above result, we give the following result:
\begin{lemma}
Let $ A \in \mathcal{B}(X,Y) $ and $ A_n := Q_n A P_n : X_n \to Y_n $, $ (\{ X_n \}_{n \in \mathrm{N}}, \{ Y_n \}_{n \in \mathrm{N}}) $ satisfies the completeness condition, that is,
\begin{equation*}
\hbox{ \raise-2mm\hbox{$\textstyle s-\lim \atop \scriptstyle {n \to \infty}$}} P_n = I_X, \hbox{ \raise-2mm\hbox{$\textstyle s-\lim \atop \scriptstyle {n \to \infty}$}} Q_n = I_Y,
\end{equation*}
where $ Q^{(n)}_{\mathcal{R}(A_n)} $ denotes the orthogonal projection mapping $ Y_n $ onto $ \mathcal{R}(A_n) $. Then there holds that
\begin{equation*}
\hbox{ \raise-2mm\hbox{$\textstyle s-\lim \atop \scriptstyle {n \to \infty}$}} Q^{(n)}_{\mathcal{R}(A_n)} Q_n  Q_{\overline{\mathcal{R}(A)}} = Q_{\overline{\mathcal{R}(A)}}.
\end{equation*}
\end{lemma}
\begin{proof}
It is sufficient to prove that
\begin{equation*}
\hbox{ \raise-2mm\hbox{$\textstyle s-\lim \atop \scriptstyle {n \to \infty}$}} Q^{(n)}_{\mathcal{R}(A_n)} Q_n y = y, \quad  \forall y \in \overline{\mathcal{R}(A)}
\end{equation*}
By (2.1) and the definition of orthogonal projection, we know that, for any $ y \in \overline{\mathcal{R}(A)} $, there exist a sequence $ \{ x_n \} $ such that $ x_n \in X_n $ and
\begin{equation*}
\Vert A P_n x_n - y \Vert \to 0, \ n \to \infty.
\end{equation*}
Then 
\begin{equation*}
\Vert Q_n A P_n x_n - Q_n y \Vert \leq \Vert Q_n \Vert \Vert  A P_n x_n - y \Vert \to 0, \ n \to \infty,
\end{equation*}
and thus, for any $ y \in \overline{\mathcal{R}(A)} $,
\begin{equation*}
\Vert  Q^{(n)}_{\mathcal{R}(A_n)} Q_n y - y \Vert \leq \Vert Q^{(n)}_{\mathcal{R}(A_n)} Q_n y - Q_n y \Vert + \Vert Q_n y - y \Vert
\end{equation*}
\begin{equation*}
\leq   \Vert Q_n A P_n x_n - Q_n y \Vert + \Vert Q_n y - y \Vert  \to 0, \ n \to \infty.
\end{equation*}
\end{proof}

\begin{proof}
We utilize the main idea in [4, Theorem 2.2 (c)] to complete the proof of our main result:
Let $ b \in Y \setminus \mathcal{R}(A) \oplus \mathcal{R}(A)^\perp $, we assume that there exist a bounded subsequence $ \{ A^\dagger_{n_k } Q_{n_k} Q_{\overline{\mathcal{R}(A)}} b \} $ of $ \{ A^\dagger_n Q_n Q_{\overline{\mathcal{R}(A)}} b \} $, that is,
\begin{equation*}
\sup_k \Vert A^\dagger_{n_k } Q_{n_k} Q_{\overline{\mathcal{R}(A)}} b \Vert  < \infty
\end{equation*}
Since Hilbert space $ X $ is reflexive, by Eberlein-Shymulan theorem, we can abstract a weakly convergent sequence from  $ \{ A^\dagger_{n_k } Q_{n_k} Q_{\overline{\mathcal{R}(A)}} b \} $, again denote by $ \{ A^\dagger_{n_k } Q_{n_k} Q_{\overline{\mathcal{R}(A)}} b \} $, that is,
\begin{equation*}
A^\dagger_{n_k } Q_{n_k} Q_{\overline{\mathcal{R}(A)}} b \stackrel{w}{\longrightarrow} x_{\infty} \in X.
\end{equation*}
then
\begin{equation}
A A^\dagger_{n_k }Q_{n_k} Q_{\overline{\mathcal{R}(A)}} b \stackrel{w}{\longrightarrow} A x_{\infty},\ Q_{n_k} A A^\dagger_{n_k } Q_{n_k} Q_{\overline{\mathcal{R}(A)}} b \stackrel{w}{\longrightarrow} A x_{\infty}.
\end{equation}
Notice that
\begin{equation*}
 Q_{n_k} A A^\dagger_{n_k } Q_{n_k} Q_{\overline{\mathcal{R}(A)}} b = Q_{n_k} A P_{n_k} A^\dagger_{n_k } Q_{n_k} Q_{\overline{\mathcal{R}(A)}} b = A_{n_k} A^\dagger_{n_k } Q_{n_k} Q_{\overline{\mathcal{R}(A)}} b
 \end{equation*}
 \begin{equation*}
= Q^{(n)}_{\mathcal{R}(A_{n_k})} Q_{n_k} Q_{\overline{\mathcal{R}(A)}} b \stackrel{s}{\longrightarrow} Q_{\overline{\mathcal{R}(A)}} b, \ k \to \infty \quad  (\textrm{by Lemma 2.2})
 \end{equation*}
This together with (2.2) gives $ Q_{\overline{\mathcal{R}(A)}} b = A x_{\infty} \in \mathcal{R}(A). $ By the decomposition $ Y =  \overline{\mathcal{R}(A)} \oplus \overline{\mathcal{R}(A)}^\perp = \overline{\mathcal{R}(A)} \oplus \mathcal{R}(A)^\perp  $,  $ b \in Y \setminus \mathcal{R}(A) \oplus \mathcal{R}(A)^\perp $ implies that $ Q_{\overline{\mathcal{R}(A)}} b \in \overline{\mathcal{R}(A)} \setminus \mathcal{R}(A) $. This is a contradiction.
\newline \indent In this way, we conclude that, any subsequence of $ \{ A^\dagger_{n_k } Q_{n_k} Q_{\overline{\mathcal{R}(A)}} b \} $ has a subsubsequence which diverge to infinity. (1.6) follows.
\end{proof}

\section{Applications}
\begin{example}
(Numerical differentiation):
\indent Set 
\begin{equation}
A \varphi(x) :=  \int^x_0  \varphi (t) dt = y(x), \ x \in (0,2\pi)
\end{equation}
It is a common and efficient way to establish stable numerical differentiation  by solving (3.1) with Petrov-Galerkin methods(See [7]).$ A \in \mathcal{C}(L^2(0,2\pi)) $ with range $ \mathcal{R}(A) = \mathcal{H}^1_0(0,2\pi) $, where
 \begin{equation*}
 \mathcal{H}^1_0(0,2\pi) := \{ \varphi \in H^1(0,2\pi): \varphi(0) = 0 \}, 
 \end{equation*}
$ H^1 $ denotes the classical Sobolev space of order $ 1 $. Its closure $ \overline{\mathcal{R}(A)} = L^2 (0,2\pi) $. Thus, application of Theorem 1.1 gives that, for $ y \in L^2(0,2\pi) \setminus \mathcal{H}^1_0(0,2\pi) $, with arbitrary $ L^2(0,2\pi) $ basis, for instance, trigonometric polynomial, Legendre polynomial, Haar functions, piecewise constant, the approximate solution $ \{ A^\dagger_n Q_n y \}  $ induced by corresponding Petrov-Galerkin methods must all diverge to infinity. This emphasize the fact that, with no proper initial value condition, that is, $ \varphi(0) = 0 $, any Petrov-Galerkin method for numerical differentiation of first order through (3.1) must fail. Thus using $ y(x) - y(0) $ to replace the function $ y $ is a must.

\end{example}

\begin{example}
Let Symm's integral equation of the first kind be 
\begin{equation}
 (K \Psi)(t) := - \frac{1}{\pi} \int^{2\pi}_0 \Psi (s) \ln \vert \boldsymbol{\gamma}(t) - \boldsymbol{\gamma}(s) \vert ds = g(t), \quad x \in [0,2\pi],
\end{equation}
 with $ \gamma(s) = (a(s), b(s)),\ s \in [0,2\pi] $ being a parametric representation of boundary $ \partial \Omega $ of $ \Omega \subset \mathrm{R}^2 $  where $ \Omega $ is some bounded, simply connected and analytic region which satisfies the following two properties
 \begin{itemize}
 \item $ \vert {\boldsymbol{\gamma}}' (s) \vert > 0 $ for all $ s \in [0,2\pi] $,
 \item there exists $ z_0 \in \Omega $ with $ \vert x - z_0 \vert \neq 1 $ for all $ x \in \partial \Omega $,
 \end{itemize}
By application of [8, Lemma 2.6 (b)] of $ s = 1 $, we know that $ K \in \mathcal{B} (L^2(0,2\pi), H^1_{per} (0,2\pi)) $ with range $ \mathcal{R}(K) = H^1_{per} (0,2\pi) $ ($ H^1_{per} (0,2\pi) $ denotes the $ 2\pi $ periodic Sobolev space, see [2,5]). Composing with compact embedding $ i :  H^1_{per} (0,2\pi) \subset \subset L^2 (0,2\pi) $, it follows that  $ K \in \mathcal{C} (L^2(0,2\pi)) $ with dense range. \newline \indent Now application of Theorem 1.1 yields that, for $ g \in L^2(0,2\pi) \setminus H^1_{per} (0,2\pi) $, with arbitrary $ L^2(0,2\pi) $ basis, the approximate solution sequence $ \{ K^\dagger_n Q_n g \} $ induced by corresponding Petrov-Galerkin method must diverge to infinity.
\newline \indent In [8] we verify (1.3) under trigonometric basis and prove that,  for $ g \in L^2(0,2\pi) \setminus  H^1_{per} (0,2\pi) $, least squares, dual least squares and Bubnov-Galerkin methods with trigonometric basis all diverge to infinity. Above divergence result is a stronger one which does not depend on the choosing basis.
\end{example}

\section{Further discussion with singular value decomposition}
\indent In this section, we strengthen $ A $ from bounded linear to be compact linear operator $ A $ with singular value system $ (\sigma_n; v_n, u_n) $. Thus, we can obtain characterizations (see [3, Chapter 2.2]) as follows
\begin{itemize}
\item \begin{equation*}
 \mathcal{R}(A)\oplus \mathcal{R}(A)^\perp = \{ y \in Y:\sum^\infty_{n=1} \frac{ \vert (y, u_n) \vert^2 }{\sigma^2_n  } <  \infty \}.
\end{equation*}
\item \begin{equation*}
 \overline{\mathcal{R}(A)} = \{ y = \sum^\infty_{n=1} y_n u_n  : \{ y_k \} \in l^2 \}.
\end{equation*}
\item
\begin{equation}
Q_{\overline{\mathcal{R}(A)}} = \sum^\infty_{n=1} \langle \cdot,\ u_n \rangle u_n.
\end{equation}
\end{itemize}
Now, in term of singular value decomposition,  Theorem 1.1 can be rewritten in a more precise way.
\begin{theorem}
Let $ A $ be a compact linear operator mapping from $ X $ to $ Y $, $ X,Y $ be all infinite-dimensional Hilbert spaces, $ \{ \xi_k \}^{\infty}_{k=1} $, $ \{ \eta_k \}^\infty_{k=1} $ be arbitrary basis of $ X $, $ Y $ respectively, subspace sequence $ \{ X_n \} $, $ \{ Y_n \} $ be subspace sequence induced by $ \{ \xi_k \}^{\infty}_{k=1} $, $ \{ \eta_k \}^\infty_{k=1} $ respectively, that is, $ X_n = \textrm{span} \{ \xi_k \}^{n}_{k=1} $, $ Y_n = \textrm{span}\{ \eta_k \}^{n}_{k=1} $. Then the following results holds:
\newline \indent (1)  If $ b \in Y $ satisfies
\begin{equation*}
\sum^\infty_{n=1} \frac{1}{\sigma^2_n} \vert ( b, u_n ) \vert^2  =  \infty, \ \textrm{that is}, \ b \notin \mathcal{R}(A)\oplus \mathcal{R}(A)^\perp,
\end{equation*}
then 
\begin{equation}
\lim_{n \to \infty} \Vert A^\dagger_n Q_n Q_{\overline{\mathcal{R}(A)}} b \Vert_{X} = \infty,
\end{equation}
where $ A_n $ and $ Q_{\overline{\mathcal{R}(A)}} $ are defined in Theorem 1.1, (4.1) respectively. 
\newline \indent (2) In particular, if  $ b = \sum^\infty_{n = 1 } b_n u_n \in Y $ satisfies
\begin{equation}
 \{ b_n \} \in l^2,  \quad
\sum^\infty_{n=1} \frac{1}{\sigma^2_n} \vert ( b, u_n) \vert^2  =  \infty, \ \textrm{that is}, \ b \in \overline{\mathcal{R}(A)} \setminus \mathcal{R}(A),
\end{equation}
 then we have $ \lim_{n \to \infty} \Vert A^\dagger_n Q_n b \Vert = \infty. $
  \newline \indent (3) If $ \{ u_n \} $ forms a complete orthogonal system of $ Y $ and $ b \in Y $ satisfies
\begin{equation}
\sum^\infty_{n=1} \frac{1}{\sigma^2_n} \vert ( b, u_n) \vert^2  =  \infty, \ \textrm{that is}, \ b \in Y \setminus \mathcal{R}(A),
\end{equation}
then we have $ \lim_{n \to \infty} \Vert A^\dagger_n Q_n b \Vert = \infty. $
\end{theorem}
As a direct application of above result to Integral equation of backward heat conduction (see [3])
\begin{example}
Let the integral equation of backward heat conduction be 
\begin{equation}
\int^\pi_0 k (x, \tau) v_0 (\tau) d \tau = f(x), \quad x \in (0,\pi),
\end{equation}
with 
\begin{equation*}
 k (x, \tau) = \frac{2}{\pi} \sum^\infty_{n=1} e^{-n^2} \sin(n \tau) \sin(nx)
\end{equation*}
where $ f $ denotes the known final temperature with $ f(0) = f(\pi) = 0 $. We need to determine the initial temperature $ v_0 (x), \ x \in [0,\pi] $.
\newline \indent (4.3) is a compact linear operator in $ L^2 (0, \pi) $ with singular value system $ (e^{-n^2}, \sqrt{\frac{2}{\pi}} \sin(nx), \sqrt{\frac{2}{\pi}} \sin(nx) ) $. Notice that $ \{ \sqrt{\frac{2}{\pi}} \sin(nx) ) \} $ form an orthogonal basis of $ L^2(0,\pi) $.  By theorem 4.1 (3), we deduce that, for arbitrary $ f \in L^2(0,\pi) $ which satifies 
\begin{equation*}
 \sum^\infty_{n=1} e^{2n^2} \vert f_n \vert^2 = \infty , \quad \textrm{where} \ f_n:= \sqrt{\frac{2}{\pi}} \int^\pi_0 f(\tau) \sin(n\tau) d\tau,
 \end{equation*}
with arbitrary $ L^2(0, \pi) $ basis, the approximate solution sequence $ \{ A^\dagger_n Q_n b \} $ induced by corresponding Petrov-Galerkin method must diverge to infinity.
\end{example}

\section*{Acknowledgement}
I would like to thank some anonymous referee for the suggestion that extend the divergence criteria from compact linear operator equation into bounded linear operator equation.
\section*{References}
\bibliographystyle{elsarticle-num-names.bst}

\end{document}